\documentclass[letterpaper, 10pt, dvipsnames]{article}

\usepackage{type1ec}
\usepackage[T1]{fontenc}
\usepackage[utf8]{inputenc}    
\usepackage[english]{babel}
\usepackage[normalem]{ulem}
\usepackage{dsfont, bm, pifont, mathrsfs}
\usepackage{stmaryrd}
\usepackage{bbm}
\usepackage{enumitem}
\usepackage{bold-extra}
\usepackage{float, graphicx, wrapfig, tikz}
\usetikzlibrary{positioning, fit}
\usepackage[margin=0.8in]{subcaption}
\usepackage[font=small, margin=0.6in]{caption}

\usepackage{amsthm, amsmath, amsfonts, amssymb, mathrsfs, mathtools}
\usepackage[alphabetic]{amsrefs}

\usepackage{hyperref, xcolor, titlesec} 
\hypersetup{colorlinks = true, urlcolor = RoyalBlue!70!Black,linkcolor=BrickRed, citecolor=ForestGreen, bookmarksopen = true}
\usepackage[nomarginpar]{geometry}
\numberwithin{equation}{section}
\usepackage[boxsize = .3em]{ytableau}
\usepackage{cleveref}

\DeclareMathOperator{\Tr}{Tr}

\newcommand{\N}{\mathbb{N}}

\newcommand{\Rbb}{\mathbb{R}}

\newcommand{\e}{\mathrm{e}}

\let\d\relax
\newcommand{\d}{\mathrm{d}}

\newcommand{\A}{\mathbf{A}}

\renewcommand{\q}{\mathbf{q}}
\renewcommand{\u}{\mathbf{u}}

\newcommand{\X}{\mathbf{X}}
\newcommand{\G}{\mathbf{G}}

\renewcommand{\epsilon}{\varepsilon}
\renewcommand{\phi}{\varphi}

\renewcommand{\P}{\mathds{P}}
\newcommand{\E}{\mathds{E}}

\newcommand{\bma}{\begin{bmatrix}}
\newcommand{\ema}{\end{bmatrix}}
\def\bet{\begin{thm}}
\def\eet{\end{thm}}
\def\bel{\begin{lem}}
\def\eel{\end{lem}}
\def\bas{\begin{ass}}
\def\eas{\end{ass}}
\def\bed{\begin{defn}}
\def\eed{\end{defn}}
\def\bep{\begin{prop}}
\def\eep{\end{prop}}
\def\beq{\begin{equation}}
\def\eeq{\end{equation}}
\def\bea{\begin{equation*}}
\def\eea{\end{equation*}}
\def\bex{\begin{ex}}
\def\eex{\end{ex}}
\def\bp{\begin{proof}}
\def\ep{\end{proof}}
\def\benr{\begin{enumerate}[label=(\roman*)]}
\def\eenr{\end{enumerate}}

\newcommand{\unn}[2]{\left[\!\left[#1,#2\right]\!\right]}

\newtheorem{ccounter}{ccounter}[section]
\newtheorem{thm}[ccounter]{Theorem}
\newtheorem{lem}[ccounter]{Lemma}

\newtheorem{defn}[ccounter]{Definition}
\newtheorem{prop}[ccounter]{Proposition}
\newtheorem{ass}[ccounter]{Assumption}
\newtheorem{ex}[ccounter]{Example}
\theoremstyle{definition}
\newtheorem{rmk}[ccounter]{Remark}

\titleformat{\section}[block]{\normalfont\filcenter}{\large\Roman{section}.}{.7em}{\large\scshape}

\titleformat{\subsection}[runin]{\normalfont}{\large \bf \thesubsection .}{.5em}{\bf}[.]
\titleformat{\subsubsection}[runin]{\normalfont}{\bf \thesubsubsection .}{.5em}{\bf}[.]
\titleformat*{\paragraph}{\itshape\mdseries}

\begin{document}
\tikzset{every node/.style={circle, minimum size=.1cm, inner sep = 2pt, scale=1}}

\title{\vspace{-5ex}\bfseries\scshape{Convergence of local eigenvector processes of generalized Wigner matrices}}
\author{L. \textsc{Benigni}\thanks{Supported in part by a Natural Sciences and Engineering Research Council of Canada (NSERC) RGPIN 2023-03882 \& DGECR 2023-00076.}\\\vspace{-0.15cm}\footnotesize{\it{Universit\'e de Montr\'eal}}\\\footnotesize{\it{lucas.benigni@umontreal.ca}}\and M. \textsc{Rezaei Feyzabady}\\\vspace{-0.15cm}\footnotesize{\it{Universit\'e de Montr\'eal}}\\\footnotesize{\it{mohammadreza.rezaei.feyzabady@umontreal.ca}}}
\date{}
\maketitle
\abstract{We prove convergence of eigenvector processes of the form $(\sqrt{N}\langle \u_k,A_t\u_k\rangle)_{t\in[0,1]}$ where $\u_k$ is a bulk eigenvector of generalized Wigner matrices and $(A_t)$ a family of symmetric matrices with bounded norm and H\"{o}lder regularity. We give explicit examples of limiting processes and prove that a large class of Gaussian process with H\"{o}lder-continuous covariance function can be obtained as such a limit using its Karhunen--Lo\`eve expansion.
The proof is based on the multi-dimensional convergence proved in \cite{benigni2022fluctuations} and a tightness criterion proved using H\"{o}lder regularity of the observables.
}

\section{Introduction and main results} The study of universal behavior in local eigenvalue and eigenvector statistics of large random matrices has witnessed remarkable progress in recent years. The modern theory originates with the seminal works \cites{erdHos2011universality, tao2011random}, which established the first proofs of universality for eigenvalue statistics of Wigner matrices—large symmetric or Hermitian random matrices with independent entries. These initial results were later refined with stronger forms of eigenvalue control, including rigidity and gap universality \cites{erdHos2012rigidity, erdHos2015gap, bourgade2016fixed}.

The understanding of eigenvectors in the delocalized phase has also advanced significantly. A fundamental step was the proof of delocalization: for any $D>0$, there exists $C>0$ such that
\[
\sup_{\q\in\mathbb{S}^{N-1}}\P\left(\sup_{k\in\unn{1}{N}} \big|\langle \q,\u_k\rangle\big| \leqslant \sqrt{\frac{C\log N}{N}} \right)\leqslant N^{-D},
\]
a result obtained through a series of works \cites{erdHos2009local, erdHos2009semicircle, isotropic1, vu2015random, isotropic2, benigni2022optimal}.

Beyond delocalization, attention has turned to finer descriptions of eigenvector entries. For a fixed $r\in\N$, given unit vectors $(\q_i)_{i=1}^r$ and indices $k_1,\dots,k_r\in\unn{1}{N}$, one expects
\[
\left( \sqrt{N}\langle\q_1,\u_{k_1}\rangle,\dots,\sqrt{N}\langle \q_r,\u_{k_r}\rangle \right) \xrightarrow[N\to\infty]{(d)} \left(\mathcal{N}_1,\dots,\mathcal{N}_r\right),
\]
where $(\mathcal{N}_i)_{i=1}^r$ are independent standard Gaussian random variables. While the general form of this entrywise convergence remains open, substantial progress has been made \cites{knowles2013eigenvector, tao2012random, bourgade2017eigenvector, marcinek2022high}.

A stronger manifestation of delocalization is encapsulated in the Eigenstate Thermalization Hypothesis (ETH). For a deterministic matrix $A$ of bounded norm, ETH asserts that for any $\varepsilon>0$ and $D>0$,
\[
\P\left( \sup_{k\in\unn{1}{N}}\Big| \langle \u_k,A\u_\ell \rangle - \delta_{k\ell}\langle A\rangle \Big| \geqslant N^{\varepsilon}\sqrt{\frac{\langle A^2\rangle}{N}} \right)\leqslant N^{-D}, \quad \text{where } \langle A\rangle=\tfrac{1}{N}\Tr(A).
\]
This conjecture has recently been confirmed in a series of works \cites{cipolloni2021eigenstate, cipolloni2022rank, cipolloni2023eigenstate} with earlier contributions for specific observables $A$ in \cites{bourgade2017eigenvector, bourgade1807random,benigni2021fermionic, benignilopatto}.

Even finer results address the fluctuations of ETH quantities. Specifically, for $k,\ell\in\unn{1}{N}$,
\[
\sqrt{\frac{N}{(1+\delta_{k\ell})\langle {\mathring{A}}^2\rangle}} \left( \langle \u_k,A\u_\ell\rangle-\delta_{k\ell}\langle A\rangle \right) \xrightarrow[N\to\infty]{(d)} \mathcal{N}(0,1), \quad \text{with }\; \mathring{A}=A-\langle A\rangle \mathrm{Id}_N,
\]
as proved in \cites{cipolloni2022normal, benigni2022fluctuations, benigni2023fluctuations, he2025extremal}.
In the present work, we remain within this framework and investigate fluctuations at the process level. More precisely, we study the convergence in the space of c\`adl\`ag functions on $[0,1]$ of the process
\begin{equation}\label{eq:defX}
\X_{k,\ell}^{A,N} = \X_{k,\ell}^{\A} \;=\; \left(\sqrt{\frac{N}{1+\delta_{k\ell}}}\,\langle \u_k,A_t \u_\ell\rangle\right)_{t\in[0,1]}
\end{equation}
for a family of traceless symmetric matrices $\A = (A_t)_{t\in[0,1]}$ with bounded norm.

The study of eigenvector processes has its origins in \cite{silverstein1981describing}, where the author established several properties of eigenvectors that strongly suggested a close connection with Haar-distributed eigenbases on the orthogonal group, reflecting a robust form of delocalization. However, the processes considered in \cite{silverstein1981describing} are inherently \emph{global}: they involve sums over a macroscopic number of eigenvectors, rather than sums of projections of a finite family of eigenvectors. As a consequence, the property suggested there relies on a four-moment universality principle, which is indispensable for handling such global quantities. By contrast, the processes we study here are \emph{local} in nature, depending only on finitely many eigenvectors. This distinction allows us to dispense with any additional moment assumptions on the distribution of the matrix entries, in particular the fourth-moment condition required in the global setting. Later works include \cite{donati2012truncations} which proves the following convergence of the bivariate eigenvector process for eigenvectors of the Gaussian Orthogonal or Unitary Ensemble, 
\[
\left(\sum_{\alpha=1}^{\lfloor Ns\rfloor}\sum_{k=1}^{\lfloor Nt \rfloor}\left(\vert \u_k(\alpha)\vert^2-\frac{1}{N}\right)\right)_{s,t\in[0,1]}
\xrightarrow[N\to\infty]{(d)} \mathrm{BBB}
\]
where $\mathrm{BBB}$ is the so-called bivariate Brownian bridge. A universality result for this process has been obtained in \cite{benaych2012universality} for Wigner matrices whose fourth moment equals the fourth-moment of a Gaussian random variable. They also prove interestingly in \cite{benaych2012universality}, that if the fourth moment is different, the limit is not given by the bivariate Brownian bridge but by a different process showing the lack of universality for this \emph{global} process. Finally, in \cite{bao2014universality}, the authors consider for a fixed deterministic vector $\q\in\mathbb{S}^{N-1}$ such that $\Vert \q\Vert_\infty\to 0$, the convergence 
\begin{equation}\label{eq:dualconvergence}
\left(
\sqrt{\frac{N}{2}}\sum_{k=1}^{\lfloor Nt\rfloor}\left(
    \langle \q,\u_k\rangle^2 -\frac{1}{N}
\right)
\right)_{t\in[0,1]}\xrightarrow[N\to\infty]{(d)}\mathrm{BB}
\end{equation}
where $\mathrm{BB}$ is the Brownian bridge under the same fourth moment assumption.

We consider in this paper $N\times N$ generalized Wigner matrices given by the following definition 
\begin{defn}[Generalized Wigner matrices]
    $W$ is a generalized Wigner matrix if it is a $N\times N$ real symmetric or complex Hermitian matrix such that its entries $(w_{ij})_{1\leqslant i\leqslant j\leqslant N}$ are independent, centered, of variance $s_{ij}$ such that there exist $c$ and $C>0$ such that for all $i,j\in\unn{1}{N}$,
    \[
    \frac{c}{N}\leqslant s_{ij}\leqslant \frac{C}{N}\quad \text{and}\quad \sum_{i=1}^N s_{ij}=1.
    \]
\end{defn}
To ease notation and readability, we restrict attention to real symmetric matrices in this note. We now give our two main results 
\begin{thm}\label{theo:main1}
    Let $p\in\N$ and $\alpha\in(0,1)$, for $i\in\unn{1}{p}$ consider $\A_i=(A_{i,t})_{t\in[0,1]}$ be a family of traceless symmetric matrices with bounded norm and such that for any $t\in(0,1]$ there exists $\delta>0$ such that $\langle A_{i,t}^2\rangle \geqslant N^{-1+\delta}$ and such that $A_{i,0}=0$. Let $\{(k_i,\ell_i)\}_{i=1}^p$ be a set of distinct pairs of indices such that for any $i\in\unn{1}{p}$, $k_i,\ell_i\in\unn{\alpha N}{(1-\alpha)N}$. Suppose that there exists $N_0\in \N,$ such that for any $i\in\unn{1}{p},$ there exists $\gamma_i>0$ and $L_i>0$, such that for any $N\geqslant N_0$, and any $t,s\in[0,1]$, 
    \[
    \langle 
        \left(
            A_{i,t}-A_{i,s}
        \right)^2
    \rangle
    \leqslant L_i\vert t-s\vert^{\gamma_i}
        \quad\text{and}\quad 
    \langle A_{i,s}A_{i,t}\rangle \xrightarrow[N\to\infty]{} K_i(t,s)    
    \]
    for some continuous positive-type\footnote{We recall that $C_i$ is of positive type if for any $q\in\N$, $t_1,\dots,t_q\in[0,1],$ and any $a_1,\dots,a_q\in \Rbb$ we have $\sum_{k,\ell=1}^qa_ka_\ell C_i(t_k,t_\ell)\geqslant 0$} function $K_i$. Then we have the convergence 
    \[
    \left(
        \X_{k_1,\ell_1}^{\A_1},\dots, \X_{k_p,\ell_p}^{\A_p}
    \right)
    \xrightarrow[N\to\infty]{(d)} \left(
        \G_1,\dots,\G_p
    \right)
    \]
    where $\X_{k,\ell}^\A$ is defined in \eqref{eq:defX} and $(\G_i)_{i=1}^p$ is an independent family of centered Gaussian processes with respective covariance $K_i$ and such that $G_{i}(0)=0$.  
\end{thm}
We remark that we put $A_0=0$ for convenience but we would actually simply need that the family $\{X^{\A_i}_{k_i,\ell_i}(0)\}$ is tight for our tightness criterion.
The proof of this theorem follows the classical approach for establishing convergence of stochastic processes, combining the multidimensional convergence framework of \cite{benigni2022fluctuations} with a tightness criterion. We also note that we give our result for bulk eigenvectors but we believe the result to be true for any eigenvector in the context of generalized Wigner matrices. However, the equivalent of Theorem \ref{theo:benignicipolloni} is not available at the edge, Gaussian fluctuations in the ETH have been studied in \cites{benigni2023fluctuations,he2025extremal} but not jointly for the same eigenvectors and different observables $A$ which is what we need here to compute the finite-dimensional distribution of the process.

We also give a proposition stating that we can represent a large class of centered Gaussian process with H\"{o}lder-continuous covariance on $[0,1]$ as the limit of a process $\X_{k,\ell}^\A$. For $\G$ a centered Gaussian process with continuous covariance $K$ on $[0,1]$ we have that $K$ is a Mercer kernel. Thus, if $(\psi_k)_{k\geqslant 1}$ is an orthonormal basis of $L^2([0,1])$ of eigenfunctions in the sense that 
    \[
    \int_0^1 K(s,t)\psi_k(s)\d s = \lambda_k\psi_k(t)
    \] 
    then the Karhunen--Lo\`eve decomposition of $\G$ is the convergence in $L^2$ and uniform in $t\in[0,1]$ of the series 
    \[
    G_t = \sum_{k=1}^{+\infty} Z_k \psi_k(t)\quad\text{with}\quad Z_k = \int_0^1 G_s \psi_k(s)\d s
    \]
    and we have that for $k,\ell \in\N$,
    $\E\left[Z_k\right]=0$ and $\E\left[Z_kZ_\ell\right]=\delta_{k\ell}\lambda_k.$ In particular, this gives the decomposition of the covariance 
    \[
    K(s,t) = \E\left[G_tG_s\right] = \sum_{k\geqslant 1}\lambda_k\psi_k(s)\psi_k(t).
    \]
    The next proposition states that any centered Gaussian process with a mild regularity assumption on its Karhunen--Lo\`eve decomposition can be represented as a limiting eigenvector process.
\begin{prop}\label{prop:main2}
Let $\alpha\in(0,1)$. Let $\G$ be a centered Gaussian process with continuous covariance $K$ on $[0,1].$ If $(\lambda_k,\psi_k)$ is the Mercer decomposition of $K$, we suppose that 
\[
\sum_{k\geqslant 1}\lambda_k\Vert \psi_k\Vert_\infty^2 <\infty    
\]
and that there exists $L,\gamma>0$ such that for every $s,t\in[0,1],$
\begin{equation}\label{eq:covlip}
\vert K(s,s)+K(t,t)-2K(s,t)\vert \leqslant L\vert t-s\vert^\gamma.
\end{equation}
Then there exists $\A$ a family of traceless symmetric matrices of bounded spectral norm such that for any $k,\ell\in\unn{\alpha N}{(1-\alpha)N}$ we have 
\[
\X_{k,\ell}^\A \xrightarrow[N\to\infty]{(d)}\G.
\]
\end{prop}  
We note that the condition on the Mercer decomposition is quite flexible as most usual Gaussian processes, such as Brownian motion, Brownian bridge, Ornstein--Uhlenbeck processes, fractional Brownian motion follow it. The proof is based on the construction of the matrix process $\A$ using the Karhunen--Lo\`eve expansion of the process $\G$. While this construction is not necessarily explicit, we give examples of explicit construction for usual Gaussian processes in Section \ref{sec:examples}.
\section{Proof of mains results}
We start by giving the main result of \cite{benigni2022fluctuations} which gives the corresponding multidimensional convergence of $\X$. If one first defines 
\[
X_{k,\ell}^A = \sqrt{\frac{N}{1+\delta_{k\ell}}}\langle \u_k,A\u_\ell\rangle.
\]
\begin{prop}[\cite{benigni2022fluctuations}*{Theorem 2.2}]\label{theo:benignicipolloni}
Let $p\in\N$, $\alpha\in(0,1)$ and a small $\delta>0.$ For any real symmetric deterministic traceless matrices $A_1,\dots,A_p$ such that $\Vert A_i\Vert \leqslant C$ and $\langle A_i^2\rangle \geqslant N^{-1+\delta}$, and any $k_1,\dots,k_{p},\ell_1,\dots,\ell_p\in\unn{\alpha N}{(1-\alpha)N}$, the family 
\[
\left(
    \sqrt{
        \frac{1}{\langle A_1^2\rangle}
    }
    X_{k_1,\ell_1}^{A_1},
    \dots,
    \sqrt{
        \frac{1}{\langle A_p^2\rangle}
    }
    X_{k_p,\ell_p}^{A_p}
\right)
\]
is approximated in the sense of moments by a centered Gaussian field $(G_{k_1,\ell_1}^{A_1},\dots,G_{k_p,\ell_p}^{A_p})$ with covariance given by 
\begin{equation}\label{eq:moment}
\E\left[
    G_{i,j}^A G_{k,\ell}^B
\right]
=
\frac{\left(
    \delta_{ki}\delta_{\ell j}+\delta_{kj}\delta_{\ell i}
\right)}
{
\sqrt{(1+\delta_{ij})(1+\delta_{k\ell})}
}
\frac{
    \langle AB\rangle 
}{
    \sqrt{\langle A^2\rangle\langle B^2\rangle}.
}
\end{equation}
\end{prop}
\begin{rmk}
We note that \cite{benigni2022fluctuations}*{Theorem 2.2} has been written in the context of Wigner matrices and not generalized Wigner matrices. However, the only input needed for the proof to work for generalized Wigner matrices is the rank-optimal multi-resolvent local laws and ETH which have been later developed for this more general class of matrices in \cites{adhikari2024eigenstate,erdos2024eigenstate}. \cite{erdos2024eigenstate} considers an even more general class of matrices but the equivalent of the theorem above would be different for these matrices as it would depend on the variance profile of the Wigner-type matrix.
\end{rmk}
We are now ready to prove Theorem \ref{theo:main1}.
\begin{proof}[Proof of Theorem \ref{theo:main1}] To prove convergence of processes we are using the multidimensional convergence and tightness properties of the family. To do so we start by consider for $i\in\unn{1}{p}$, an integer $j_i\in \N$ and a family of time $t^{(i)}_{1},\dots,t^{(i)}_{j_i}\in[0,1]$ and we thus want to consider convergence in distribution of the vector 
    \[
    \left(
        X^{\A_1}_{k_1,\ell_1}\left(
            t^{(1)}_{1}
        \right),
        \dots,
        X^{\A_1}_{k_1,\ell_1}\left(
            t^{(1)}_{j_1}
        \right),
        X^{\A_2}_{k_2,\ell_2}\left(
            t^{(2)}_{1}
        \right),
        \dots,
        X^{\A_p}_{k_p,\ell_p}\left(
            t^{(p)}_{j_p}
        \right)
    \right) .
    \]
Since we know that for all $i\in\unn{1}{p}$ and $j\in\unn{1}{j_i}$ we have the convergence 
\[\langle A_{i,t^{(i)}_{j}}^2\rangle \xrightarrow[N\to\infty]{} K_i\left(t^{(i)}_{j},t^{(i)}_{j}\right)
\] we know that for $N$ large enough there exist $c,C>0$ such that $c\leqslant \langle A_{i,t^{(i)}_{j}}^2\rangle\leqslant C.$ In particular, the renormalization by $\sqrt{\langle A_i^2\rangle}$ in Theorem \ref{theo:benignicipolloni} do not change any moment convergence and we do obtain that the moment are approximated by the Gaussian field with covariance from \eqref{eq:moment} multiplied by $\sqrt{\langle A^2\rangle\langle B^2\rangle}$. We thus obtain the multidimensional projection of independent centered Gaussian processes starting at 0 and for a fixed $i\in\unn{1}{p}$, the covariance obtained for $\X^{\A_i}_{k_i,\ell_i}$ is 
\[
\langle A_{i,t}A_{i,s}\rangle \xrightarrow[N\to\infty]{} K_i(s,t).
\] 
Since the Gaussian distribution is determined by its moments we obtain convergence in distribution from this moment convergence and we obtain the first part of the proof.

We now prove tightness of the family of processes. We note that we consider a finite family of processes and can thus prove tightness of one process $\X^{\A_i,N}_{k_i,\ell_i}$ for $i\in\unn{1}{p}$. For this reason, we omit the subscript $i$ in the rest of proof. We use the following tightness criterion from \cite{billingsley}, $\{\X^{\A,N}_{k,\ell}\}_N$ is tight if $\{X^{\A,N}_{k,\ell}(0)\}$ is tight and if for every $\varepsilon,\eta>0$ there exists $\delta\in(0,1)$ and $N_1\in\N$ such that for all $N\geqslant N_1$ and $t\in[0,1]$ we have 
\[
\frac{1}{\delta}\P\left(
    \sup_{t\leqslant s\leqslant (t+\delta)\wedge 1}
    \left\vert   
        X^{\A,N}_{k,\ell}(s)-X^{\A,N}_{k,\ell}(t)
    \right\vert>\varepsilon
\right)\leqslant \eta.
\]
Since we suppose that $A_0=0$, the first part of the criterion is immediate. For the second step, fix arbitrary $\varepsilon,\eta>0$, a point $t\in[0,1]$, and choose $\delta\in(0,1)$ (to be specified later).  
Let $j\in\mathbb{N}$ and select a parameter $\beta>0$, whose value will also be chosen later. We consider the dyadic mesh at scale $j$ and a scale parameter,
\[
\mathcal{T}_{j}
=
\left\{
s_n^{(j)}\coloneqq t+\frac{n\delta}{2^j},n\in\unn{0}{2^j}
\right\}
\quad \text{and}\quad 
\varepsilon_j = 2^\beta(1-2^{-\beta}){2^{-\beta j}}\varepsilon.
\]
We also define the family of events for $j\in \N$ and $n\in\unn{0}{2^j}$,
\[
E^{(j)}_n = \left\{\left\vert   
        X^{\A,N}_{k,\ell}\left(s^{(j)}_{n+1}\right)
        -
        X^{\A,N}_{k,\ell}\left(s^{(j)}_{n}\right)
    \right\vert
    >\varepsilon_j  
\right\},
\]For $n\in\unn{0}{2^j},$ and $q\in \N$ chosen later, we have 
\[
\P\left(
    E^{(j)}_n
\right)
\leqslant \frac{
    \E\left[
        \left(
        X^{\A,N}_{k,\ell}\left(s^{(j)}_{n+1}\right)
        -
        X^{\A,N}_{k,\ell}\left(s^{(j)}_{n}\right)
    \right)^{2q}
    \right]
}
{\varepsilon_j^{2q}}.
\]
But we have 
\[
X^{\A,N}_{k,\ell}\left(s^{(j)}_{n+1}\right)
        -
X^{\A,N}_{k,\ell}\left(s^{(j)}_{n}\right)
=
\sqrt{\frac{N}{1+\delta_{k\ell}}}\left\langle \u_k,
\left(
    A_{s^{(j)}_{n+1}}-A_{s^{(j)}_{n}}
\right)\u_\ell\right\rangle.
\]
Since $\Delta^{(j)}_n\coloneqq A_{s^{(j)}_{n+1}}-A_{s^{(j)}_{n}}$ follows the assumptions of Theorem \ref{theo:benignicipolloni} and of Theorem \ref{theo:main1}, we know that if $C>1$ there exists $N_1\in \N$ such that for all $N\geqslant N_1$ we have 
\[
\P\left(
    E^{(j)}_n
\right)
\leqslant    
C\frac{
    \left\langle 
        {\Delta^{(j)}_n}^2
    \right\rangle^q
}
{\varepsilon_j^{2q}}
\leqslant    
\frac{C_qL^q\delta^{\gamma q}}{2^{2q\beta}(1-2^{-\beta})^{2q}\varepsilon^{2q}}2^{jq\left(2\beta -\gamma \right)}.
\]
Now, at scale $j$, we see that there exists $2^j+1$ points in the mesh $\mathcal{T}_j$ in $[t,t+\delta]$, so that by a union bound, we get that for $N\geqslant N_1$,
\[
\P\Bigg(
   \bigcup_{s^{(j)}_n\in[t,t+\delta]} E^{(j)}_n
\Bigg)
\leqslant 
\frac{C_qL^q\delta^{\gamma q}}{2^{2q\beta}(1-2^{-\beta})^{2q}\varepsilon^{2q}}2^{jq\left(2\beta -\gamma +\frac{1}{q}\right)}.
\]
We now choose $\beta$ small enough and $q$ large enough so that $2\beta-\gamma+\frac{1}{q}<0$, for instance $q \geqslant \lceil \frac{2}{\gamma}\rceil$ and $\beta \leqslant \frac{\gamma}{5}.$ By another union bound, we get that 
\begin{align*}
\P\Bigg(
    \bigcup_{j\in \N} 
    \bigcup_{s^{(j)}_n\in[t,t+\delta]}
    E^{(j)}_n
\Bigg)
\leqslant \frac{C_qL^q\delta^{\gamma q}}{2^{2q\beta}(1-2^{-\beta})^{2q}\varepsilon^{2q}}\sum_{j\geqslant 0}2^{jq\left(2\beta -\gamma +\frac{1}{q}\right)}
=
\frac{C_{\gamma,L}}{\varepsilon^{2q}}\delta^{\gamma q}.
\end{align*}

We now work on the complementary of this event, namely on the event 
\[
F=\bigcap_{j\in \N}\bigcap_{s_n^{(j)}\in[t,t+\delta]}
    \left\{\left\vert   
        X^{\A,N}_{k,\ell}\left(s^{(j)}_{n+1}\right)
        -
        X^{\A,N}_{k,\ell}\left(s^{(j)}_{n}\right)
    \right\vert \leqslant \varepsilon_j
\right\}.
\]
Let $s\in[t,t+\delta]$, and define the sequence of times 
\[
u_j = t+\left\lfloor 2^j\frac{s-t}{\delta}\right\rfloor\frac{\delta}{2^j}\in \mathcal{T}_j
\quad\text{so that }\quad 
u_j\xrightarrow[j\to\infty]{}s\quad\text{and}\quad u_0=t.
\]
We also note that 
\[
\left\lfloor 2^{j+1}\frac{s-t}{\delta}\right\rfloor \in \left\{
    2\left\lfloor 2^j\frac{s-t}{\delta}\right\rfloor,2\left\lfloor 2^j\frac{s-t}{\delta}\right\rfloor+1 
\right\}
\quad\text{which gives}\quad 
u_{j+1}\in\left\{u_j,u_j+\frac{\delta}{2^{j+1}}\right\}.
\]
Thus we can write, since we work on the event $F$ and that the dyadic meshes are nested with $u_{j+1}$ and $u_j$ being adjacent (or equal) in $\mathcal{T}_{j+1}$,
\[
\left\vert
X_{k,\ell}^{\A,N}(s)-X_{k,\ell}^{\A,N}(t)
\right\vert   
\leqslant
\sum_{j\geqslant 0}\left\vert    
    X_{k,\ell}^{\A,N}(u_{j+1})-X_{k,\ell}^{\A,N}(u_j)
\right\vert
\leqslant    
\sum_{j\geqslant 0} \varepsilon_{j+1} =\varepsilon
\]
uniformly in $s\in[t,t+\delta]$.

We finally get 
\[
\frac{1}{\delta}\P\left(
    \sup_{t\leqslant s\leqslant (t+\delta)\wedge 1}
    \left\vert   
        X^{\A,N}_{k,\ell}(s)-X^{\A,N}_{k,\ell}(t)
    \right\vert>\varepsilon
\right)\leqslant 
\frac{C_{\gamma,L}}{\varepsilon^{2q}}\delta^{\gamma q-1}.
\]
We then finally need to choose $\delta\in(0,1)$ so that 
\[
\frac{C_{\gamma,L}}{\varepsilon^{2q}}\delta^{\gamma q-1}\leqslant \eta   
\Longleftrightarrow
\delta \leqslant \left(\frac{\eta \varepsilon^{2q}}{C_{\gamma,L}}\right)^{\frac{1}{\gamma q -1}}
\]
since we chose $q\in \N$ such that $q\geqslant \lceil \frac{2}{\gamma}\rceil > \frac{1}{\gamma}$. We thus finally obtain tightness of the sequence of processes and thus convergence in distribution.
\end{proof}
We now give the proof of Proposition \ref{prop:main2}.
\begin{proof}[Proof of Proposition \ref{prop:main2}]
    Since we have the Mercer decomposition $K(s,t) = \sum_{k=1}^{+\infty}\lambda_k\psi_k(s)\psi_k(t)$, a natural candidate would be to consider the diagonal matrix $A_t$ whose diagonal entries $A_{ii,t}$ are given by 
    \[
    A_{ii,t}=\sqrt{N\lambda_i}\psi_i(t).
    \]
    Indeed, we obtain that the covariance function of $\X^{\A}_{k,\ell}$ is given by 
    \[
    \langle A_tA_s\rangle = \frac{1}{N}\sum_{i=1}^N N\lambda_i\psi_i(t)\psi_i(s) \xrightarrow[N\to\infty]{}K(s,t).
    \]
    However, we have several issues from this construction : the matrices are not traceless and they are not of bounded spectral norm which is needed to use the convergence of processes from Theorem \ref{theo:main1}. We now adjust our construction, while keeping the matrix $A$ diagonal, so that our model follows these two constraints. Let $\kappa\in(0,1)$ and define 
    \[
    T_\kappa = \sum_{k=1}^{N^\kappa}\lambda_k\Vert \psi_k\Vert_\infty^2 \leqslant \sum_{k=1}^{+\infty}\lambda_k\Vert \psi_k\Vert_\infty^2<\infty.
    \]
    Consider 
    \[
    \omega_i^{(\kappa)} = \frac{\lambda_i\Vert \psi_i\Vert^2_\infty}{T_\kappa},\quad\text{in particular }\sum_{i=1}^{N^\kappa}\omega_i^{(\kappa)}=1.
    \]
    Define $n_i^{(\kappa)} = \left\lfloor\frac{N}{2}\omega_i\right\rfloor,$ so that up to rounding errors we have $\sum_{i=1}^{N^{\kappa}}2n_i^{(\kappa)}=N,$ and consider the diagonal matrix $A$ such that for $i\in\unn{1}{N^\kappa}$, with the convention that $n_0^{(\kappa)}=0$,
    \begin{align*}
    A_{j,t} = \frac{\sqrt{T_\kappa}}{\Vert \psi_i\Vert_\infty}\psi_i(t) 
    \quad \text{si }
    j\in\unn{\sum_{\ell=0}^{i-1}2n_{\ell-1}^{(\kappa)}+1}{\sum_{\ell=0}^{i-1}2n_{\ell-1}^{(\kappa)}+n_i^{(\kappa)}}
    \text{ for some }\ell\in\unn{1}{N^\kappa}\\
    A_{j,t} = -\frac{\sqrt{T_\kappa}}{\Vert \psi_i\Vert_\infty}\psi_i(t)
    \quad \text{si }
    j\in\unn{\sum_{\ell=0}^{i-1}2n_{\ell-1}^{(\kappa)}+n_i^{\kappa}+1}{\sum_{\ell=0}^{i}2n_{\ell-1}^{(\kappa)}}
    \text{ for some }\ell\in\unn{1}{N^\kappa}
    \end{align*}
    By construction, we see that $A$ is symmetric since diagonal. If $N$ is even, the matrix is exactly traceless since every term appears twice with a different sign. If $N$ is odd, then one can slighty modify the matrix for it be completely traceless, by for instance leaving the last diagonal entry to be zero. We know need to check the covariance function given by these matrices and we have 
    \[
    \langle A_tA_s\rangle 
    =
    \frac{1}{N}\sum_{i=1}^N A_{i,t}A_{i,s}
    =
    \frac{1}{N}\sum_{i=1}^{N^\kappa}2n^{(\kappa)}_i \frac{T_\kappa}{\Vert \psi_j\Vert_\infty^2}\psi_i(t)\psi_i(s)
    =
    \frac{1}{N}\sum_{i=1}^{N^\kappa}\left(
        \frac{N\lambda_i\Vert\psi_i\Vert_\infty^2}{T_\kappa}+\xi_i
    \right)
    \frac{T_\kappa}{\Vert \psi_i\Vert_\infty^2}\psi_i(t)\psi_i(t)
    \]
    with $\xi_i\in [0,2)$ corresponding to the rounding error in the definition of $n_i^{(\kappa)}$. Thus we get that 
    \[
    \langle A_tA_s\rangle 
    =
    \sum_{i=1}^{N^\kappa}\lambda_i\psi_i(t)\psi_i(s) + \frac{1}{N}\sum_{i=1}^{N^\kappa}\xi_i\frac{T_\kappa}{\Vert \psi_i\Vert_\infty^2}\psi_i(t)\psi_i(s).
    \]
    Since $\xi_i\in(0,2), \kappa\in(0,1)$ and $T_\kappa \leqslant \sum_{i=1}^\infty \lambda_i\Vert \psi_i\Vert^2<\infty,$ we obtain that 
    \[
    \langle A_tA_s\rangle \xrightarrow[N\to\infty]{} \sum_{i=1}^\infty \lambda_i\psi_i(t)\psi_i(s) = K(s,t).
    \]
    Besides, since 
    \[
    \langle (A_t-A_s)^2\rangle = \langle A_t^2\rangle + \langle A_s^2\rangle -2\langle A_tA_s\rangle
    \]
    and the assumption \eqref{eq:covlip} on the covariance function $K$,  we have that for $N$ large enough,
    \[
    \langle (A_t-A_s)^2\rangle \leqslant L\vert t-s\vert^\gamma.
    \]
    Finally, we need to see that $A_t$ has bounded spectral norm. But we see easily that for every $j\in\unn{1}{N}$, there exists $i\in\unn{1}{N^\kappa}$ such that 
    \[
    \vert A_{j,t}\vert = \frac{\sqrt{T_\kappa}}{\Vert \psi_i\Vert_\infty}\vert \psi_i\vert \leqslant \sqrt{T_\kappa}\leqslant \sqrt{\sum_{k=1}^\infty \lambda_k\Vert \psi_k\Vert_\infty^2}<\infty
    \]
    by assumption. Thus, we obtain the correct finite-dimensional distribution and tightness from the proof of Theorem \ref{theo:main1} and the corresponding convergence.
\end{proof}
\section{Examples of limiting processes}\label{sec:examples}
In this section, we present examples of families of matrices $\A$ and their corresponding limiting Gaussian processes. The proof of Proposition \ref{prop:main2} in fact provides an explicit (though dependent on the Karhunen--Loève decomposition) procedure for constructing $\A$ from a given process $\G$. Before turning to that construction, however, we first examine several more natural families of matrices $\A$ that lead to specific limiting processes.
\paragraph{Brownian bridge.} The first convergence is the dual convergence of \eqref{eq:dualconvergence} illustrated in Figure \ref{fig:BB}. Given an orthonormal basis of unit $(\q_\alpha)_{\alpha=1}^N$, we have the convergence, for $k,\ell\in\unn{\alpha N}{(1-\alpha)N},$
\begin{equation}\label{eq:brownianbridge}
\left(
    \sqrt{\frac{N}{(1+\delta_{k\ell})}}
    \sum_{\alpha=1}^{\lfloor Nt\rfloor}
    \left(
        \langle \q_\alpha,\u_k\rangle\langle \q_\alpha,\u_\ell\rangle - \frac{\delta_{k\ell}}{N}
    \right)
\right)_{t\in[0,1]}
\xrightarrow[N\to\infty]{(d)}\mathrm{BB}.
\end{equation}
We can construct the family of matrices,
\[
A_{ab,t} = \sum_{\alpha=1}^{\lfloor Nt\rfloor}
\left(
    q_\alpha(a)q_\alpha(b) - \frac{\delta_{ab}}{N}
\right)
\quad\text{then}\quad 
\langle \u_k,A_t\u_\ell\rangle = \sum_{\alpha=1}^{\lfloor Nt\rfloor}
\left(
    \langle \q_\alpha,\u_k\rangle\langle \q_\alpha,\u_\ell\rangle-\frac{\delta_{k\ell}}{N}
\right).
\]
This corresponds to a family of traceless symmetric matrices with bounded spectral norm. We can compute the autorrelation function for $s\leqslant t$,
\begin{equation}\label{eq:famvect}
    \begin{aligned}
\langle A_tA_s\rangle 
&=
\frac{1}{N}\sum_{a,b=1}^N
\sum_{\alpha=1}^{\lfloor Nt\rfloor}
\sum_{\beta=1}^{\lfloor Ns \rfloor}
\left(
    q_\alpha(a)q_\alpha(b)-\frac{\delta_{ab}}{N}
\right)
\left(
    q_\beta(a)q_\beta(b)-\frac{\delta_{ab}}{N}
\right)\\
& =
\frac{1}{N}\sum_{\alpha=1}^{\lfloor Ns\rfloor}
\Vert \q_\alpha\Vert^2
+
\frac{1}{N}\sum_{\alpha\neq \beta}\langle \q_\alpha,\q_\beta\rangle^2
-
\frac{\lfloor Nt\rfloor}{N^2}\sum_{\beta=1}^{\lfloor Ns\rfloor}\Vert \q_\beta\Vert^2 
-
\frac{\lfloor Ns\rfloor}{N^2}\sum_{\alpha=1}^{\lfloor Nt\rfloor}\Vert\q_\alpha\Vert^2
+
\frac{\lfloor Nt\rfloor\lfloor Ns\rfloor}{N^2}\\
&=
\frac{\lfloor Ns\rfloor}{N}-\frac{\lfloor Nt\rfloor\lfloor Ns\rfloor}{N^2}
\xrightarrow[N\to\infty]{} s(1-t)=K_{\mathrm{BB}}(s,t).
\end{aligned}
\end{equation}
We note that while the construction \eqref{eq:brownianbridge} is natural as it corresponds to similar \emph{global} eigenvector processes studied in \cites{silverstein1981describing,donati2012truncations,benaych2012universality,bao2014universality}, we can obtain Brownian bridge as the limiting process by constructing $\A$ as in the proof of Proposition \ref{prop:main2} since the Karhunen--Lo\`eve decomposition of the Brownian bridge is explicit, namely 
\[
\lambda_k^{\mathrm{BB}} = \frac{1}{(k\pi)^2},\quad \psi_k^{\mathrm{BB}}(t) = \sqrt{2}\sin(k\pi t).
\]
\vspace{-5ex}
\begin{figure}[!ht]
        \centering
        \includegraphics[height=.225\linewidth]{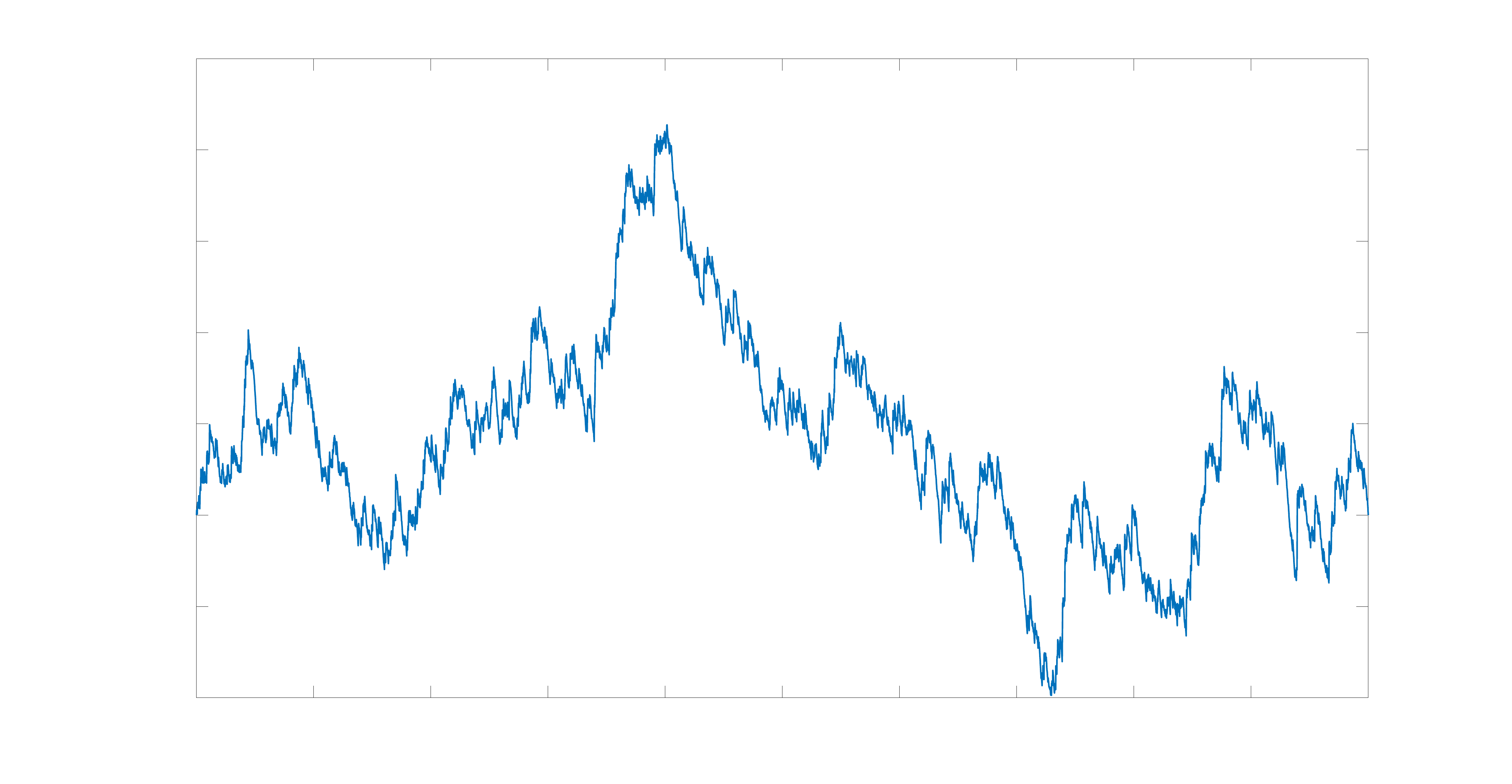}
        \includegraphics[width=.3\linewidth]{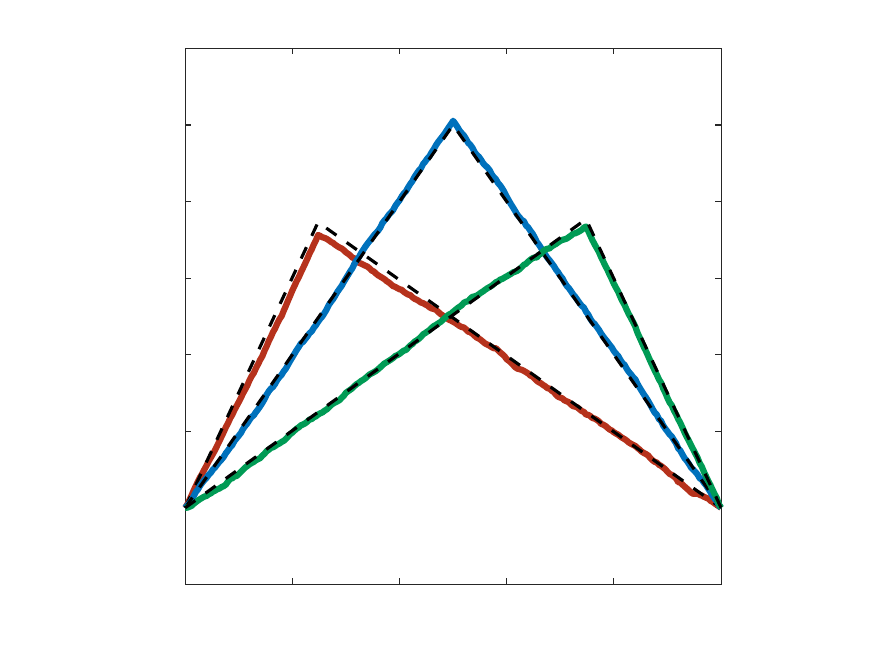}
    \caption{Sample of $\X_{k,k}^\A$ as in \eqref{eq:brownianbridge} for $\u_k$ the middle eigenvector of a Rademacher random matrix on the left. Empirical covariance functions $K(s,t)$ at times $s=\frac{1}{4},\frac{1}{2},$ and $\frac{3}{4}$ on the right.}\label{fig:BB}
\end{figure}
\paragraph{Brownian motion.} While the Karhunen--Lo\`eve decomposition of the Brownian motion is also explicit since 
\[
\lambda_k^{\mathrm{BM}} = \frac{1}{\left(k-\frac{1}{2}\right)^2\pi^2},\quad \psi_K^{\mathrm{BM}}(t) = \sqrt{2}\sin\left(\left(k-\frac{1}{2}\right)\pi t\right),
\]  we can consider the same construction as above but for a different family of vector $(\q_\alpha)$. For instance, if we consider an equiangular family of unit vectors such that for $\gamma\in\Rbb$, 
\[
\langle \q_\alpha,\q_\beta\rangle = \delta_{\alpha\beta}+(1-\delta_{\alpha\beta})\frac{\gamma}{\sqrt{N}}.
\]
Then the computation from \eqref{eq:famvect} gives us that for $s\leqslant t$,
\begin{equation}\label{eq:bm}
\langle A_tA_s\rangle = \frac{\lfloor Ns\rfloor}{N}+\frac{\lfloor Nt\rfloor\lfloor Ns\rfloor-\lfloor Ns\rfloor}{N^2}\gamma^2-\frac{\lfloor Nt\rfloor\lfloor Ns\rfloor}{N^2}\xrightarrow[N\to\infty]{}
s(1+(\gamma^2-1)t).
\end{equation}
In particular, we find that if $\gamma=1$ then the limiting covariance function $\langle A_tA_s\rangle = s\wedge t$ corresponds to Brownian motion as illustrated in Figure \ref{fig:BM}. In general, we note that this corresponds to the Gaussian process given by, if $\mathbf{B}$ is a standard Brownian motion,
\[
G_t = B_t - (1+\gamma)tB_1
\]
recovering the Brownian bridge for $\gamma=0$ and Brownian motion for $\vert \gamma\vert=1$.
\begin{figure}[!ht]
        \centering
        \includegraphics[height=.225\linewidth]{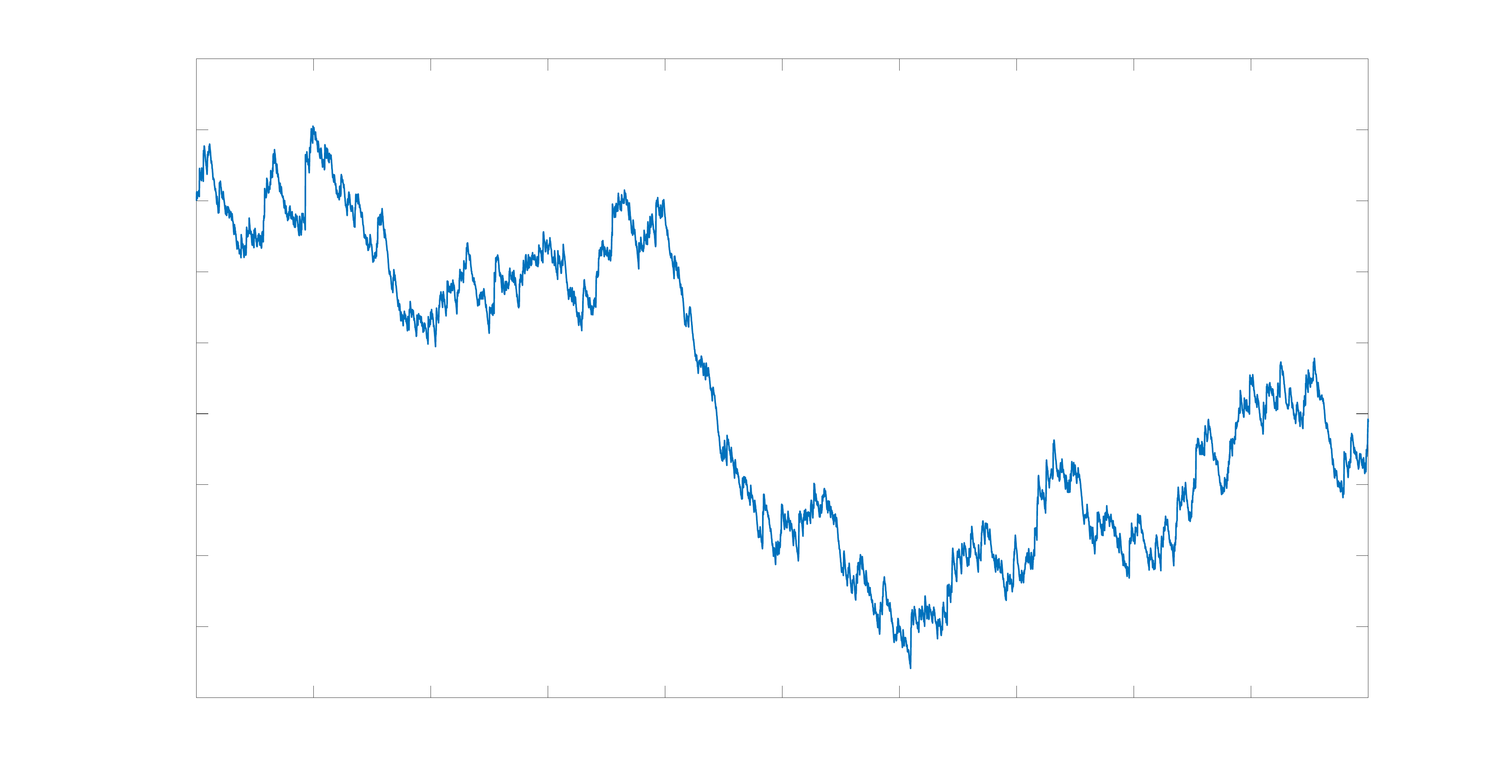}
        \includegraphics[width=.3\linewidth]{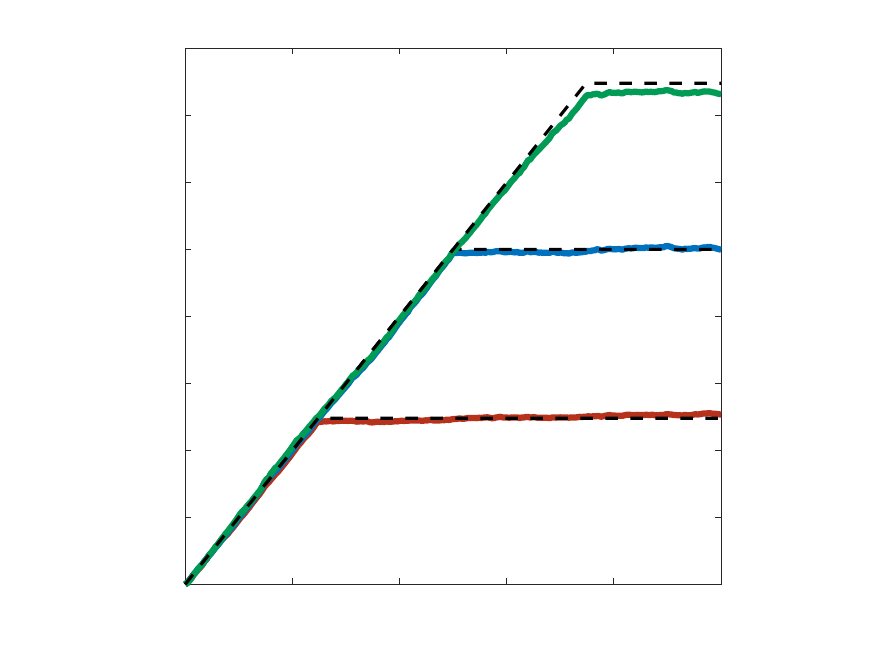}
    \caption{Sample of $\X_{k,k}^\A$ as in \eqref{eq:brownianbridge} with an equiangular family for $\u_k$ the middle eigenvector of a Rademacher random matrix on the left. Empirical covariance functions $K(s,t)$ at times $s=\frac{1}{4},\frac{1}{2},$ and $\frac{3}{4}$ on the right.}\label{fig:BM}
\end{figure}

Note that we can generalize this equiangular variant of the process by considering general unit vectors. For instance, consider a function $F:[0,1]^2\to[-1,+\infty]$, then if the matrix defined by 
\begin{equation}\label{eq:condgamma}
\Gamma_{ij} =\delta_{ij} +\frac{1-\delta_{ij}}{\sqrt{N}}
\sqrt{1+F\left(
    \frac{i}{N},\frac{j}{N}
\right)}
\end{equation}
is positive semi-definite then, using for instance the Cholesky decomposition of $\Gamma$, one can find a family of vectors $(\q_\alpha)_{\alpha=1}^N$ such that 
\[
\langle \q_\alpha,\q_\beta\rangle^2 = \delta_{\alpha\beta}+\frac{1-\delta_{\alpha\beta}}{\sqrt{N}}\left(
    1+F\left(\frac{i}{N},\frac{j}{N}\right)
\right)
\]
and thus from \eqref{eq:famvect} we get the covariance function for $s\leqslant t$,
\begin{align*}
\langle A_tA_s\rangle 
&=
\frac{\lfloor Ns\rfloor}{N}+\frac{\lfloor Ns\rfloor(\lfloor Nt\rfloor-1)}{N^2} + \frac{1}{N^2}\sum_{\alpha=1}^{\lfloor Nt\rfloor}\sum_{\beta=1}^{\lfloor Ns\rfloor}F\left(\frac{\alpha}{N},\frac{\beta}{N}\right)-\frac{\lfloor Ns\rfloor(\lfloor Nt\rfloor-1)}{N^2}\\
&\xrightarrow[N\to\infty]{}
s+\int_0^s\int_0^t F(x,y)\d x\d y\eqqcolon s+K(s,t).
\end{align*}
Thus we see that if $\mathbf{B}$ is a standard Brownian motion and if $\G$ is a centered Gaussian process with covariance function $K(s,t)$ independent of $\mathbf{B}$ then the limiting process is given by $\mathbf{B}+\G$.

The condition \eqref{eq:condgamma} can be tricky and to our knowledge should be checked for each function $F$. As an illustration, we give an example in Figure \ref{fig:BMsin} with the function 
\begin{equation}\label{eq:examplesin}
F(x,y)=\pi^2\sin(\pi x)\sin(\pi y) \quad\text{with}\quad K(s,t)= \int_0^s\int_0^tF(x,y)\d x\d y = \left(
    \cos(\pi s)-1
\right)\left(
    \cos(\pi t)-1
\right).
\end{equation}
\vspace{-5ex}
\begin{figure}[!ht]
        \centering
        \includegraphics[height=.225\linewidth]{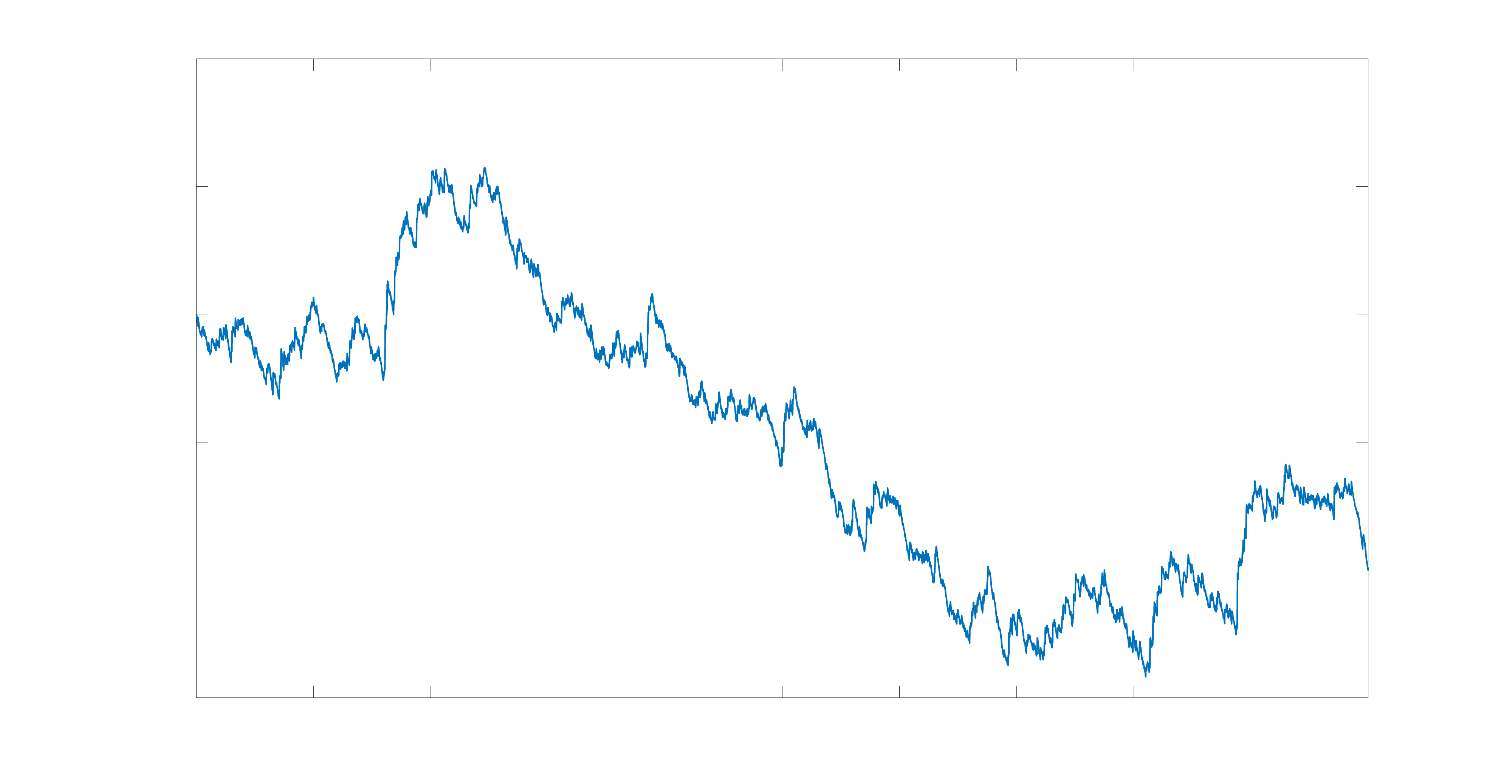}
        \includegraphics[width=.3\linewidth]{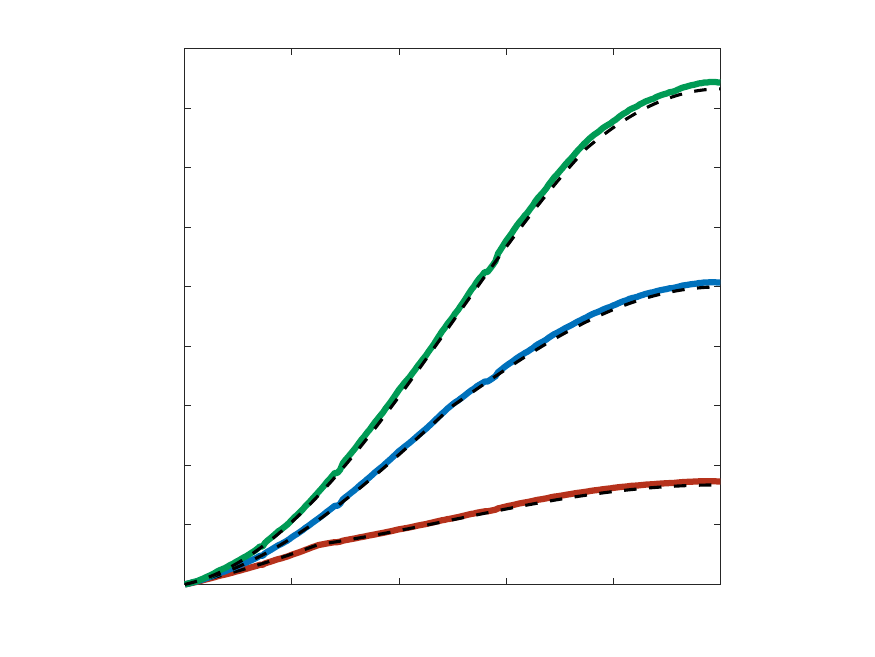}
    \caption{Sample of $\X_{k,k}^\A\xrightarrow{}\mathbf{B}+\G$ with $\G$ a centered Gaussian process with covariance given by \eqref{eq:examplesin} for $\u_k$ the middle eigenvector of a Rademacher random matrix on the left. Empirical covariance functions $K(s,t)$ at times $s=\frac{1}{4},\frac{1}{2},$ and $\frac{3}{4}$ on the right.}\label{fig:BMsin}
\end{figure}
\paragraph{Separable covariance.} We give an explicit construction so that our process $\X^\A_{k,\ell}$ converges in distribution to a centered Gaussian process with covariance function given by 
\[
K(s,t) = \int_0^\infty f_u(s)f_u(t)\d u,
\]
which is the case for some usual Gaussian processes such as the Ornstein--Uhlenbeck process. For this setting, we can construct a traceless family of observables in the following way. If $(\mathbf{e}_\alpha)_{\alpha=1}^N$ is an orthonormal basis, say the standard basis of $\Rbb^N$, then we define,
\begin{equation}\label{eq:defAOU}
A_{t} = \sum_{\alpha=1}^{\lfloor \frac{N}{2}\rfloor}f_{\frac{\alpha}{N}}(t)B_\alpha
\quad\text{with}\quad 
B_\alpha =\frac{1}{\sqrt{2}}\left( \mathbf{e}_{2\alpha}\mathbf{e}_{2\alpha}^\top-\mathbf{e}_{2\alpha-1}\mathbf{e}_{2\alpha-1}^\top\right)\quad\text{for}\quad\alpha\in\unn{1}{\left\lfloor\frac{N}{2}\right\rfloor}.
\end{equation}
Since $B_\alpha$ is traceless for every $\alpha$ then $A_t$ is too. Besides, if $\sup_{u>0}\Vert f_u\Vert_\infty<\infty$, then the spectral norm of $A_t$ is bounded and we have 
\[
\langle A_tA_s\rangle 
=
\frac{1}{N}\sum_{\alpha,\beta=1}^{\lfloor \frac{N}{2}\rfloor}
f_{\frac{\alpha}{N}}(t)f_{\frac{\beta}{N}}(s)\Tr\left(B_\alpha B_\beta\right)
=
\frac{1}{N}\sum_{\alpha=1}^Nf_{\frac{\alpha}{N}}(t)f_{\frac{\alpha}{N}}(s)
\xrightarrow[N\to\infty]{}\int_0^\infty f_u(s)f_u(t)\d u.
\] 
We illustrate this construction for the Ornstein--Uhlenbeck process in Figure \ref{fig:OU} for which we consider the family of functions $f_u(t) = \sigma\e^{-\theta (t-u)}\mathds{1}_{t\geqslant u}$ with $\sup_{u>0}\Vert f_u\Vert_\infty = 1$. Indeed,
\[
\int_0^\infty f_u(s)f_u(t)\d u = \frac{\sigma^2}{2\theta}\left(
\e^{-\theta\vert t-s\vert}-\e^{-\theta(t+s)}
\right)\eqqcolon K_{\mathrm{OU}}(s,t).
\]
\vspace{-5ex}
\begin{figure}[!ht]
        \centering
        \includegraphics[height=.19\linewidth]{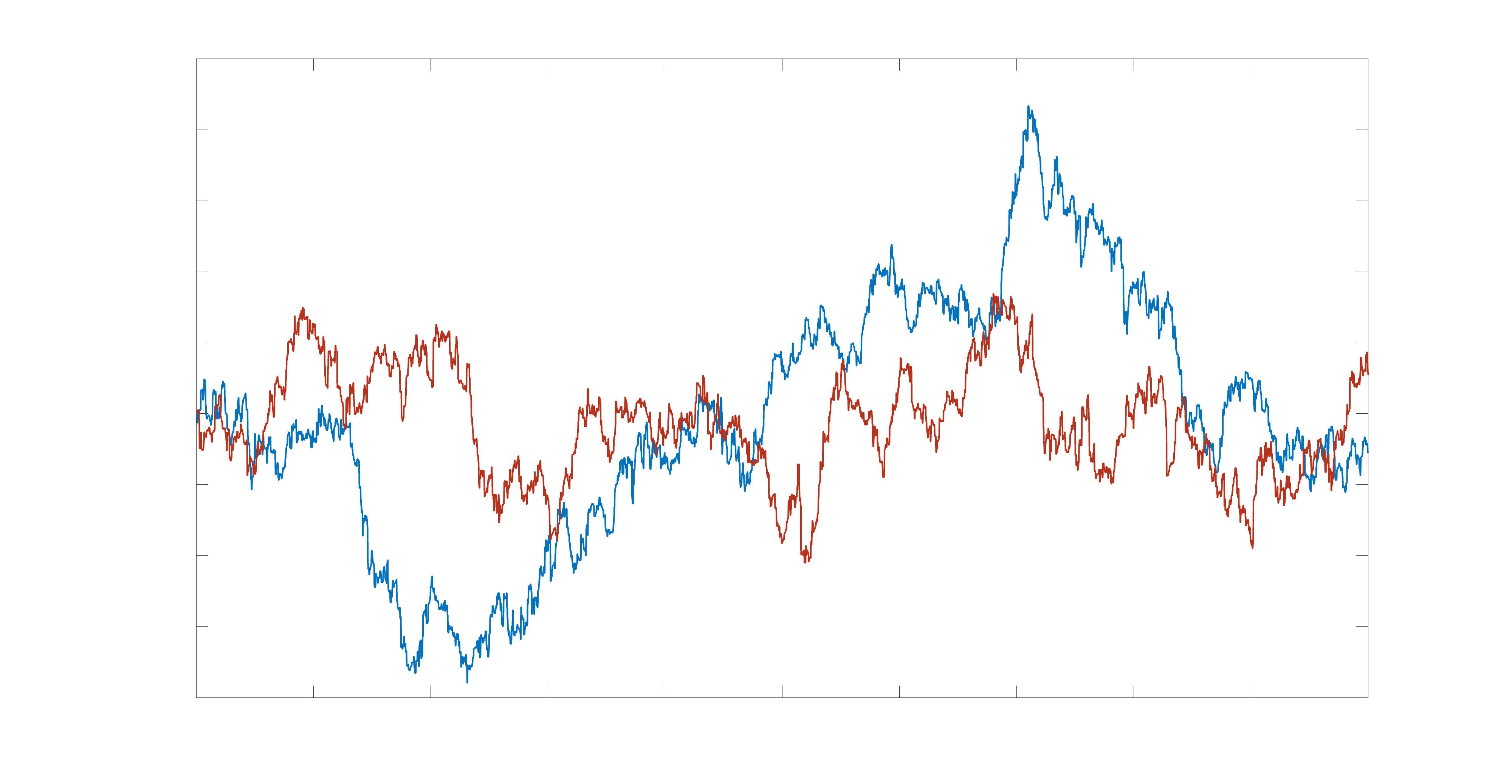}
        \includegraphics[width=.25\linewidth]{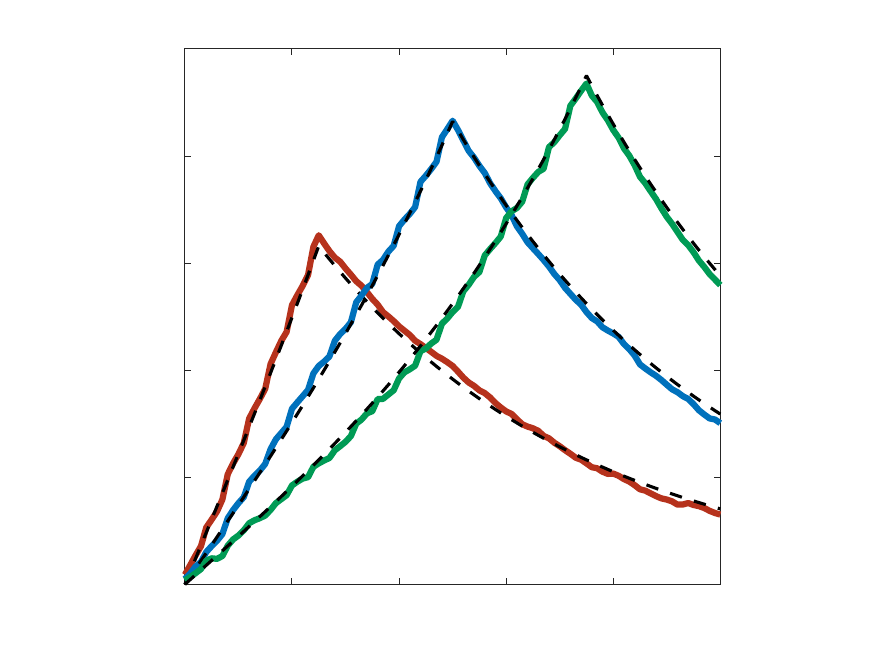}
        \includegraphics[width=.25\linewidth]{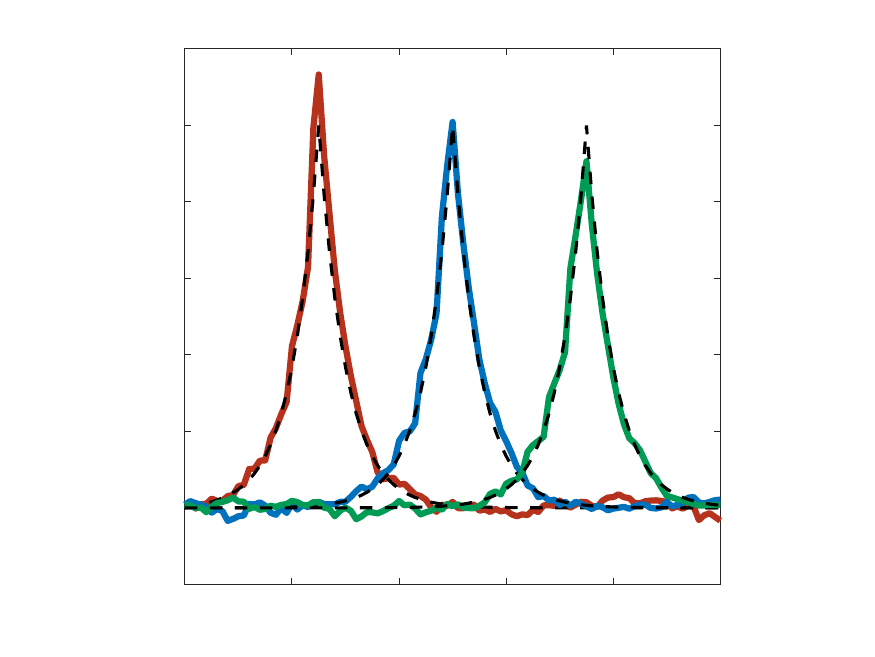}
    \caption{Sample of $\X_{k,k}^\A$ for $\A$ defined as in \eqref{eq:defAOU} for $(\theta,\sigma)=(2,1)$ in blue and $(\theta,\sigma)=(20,1)$ in red for $\u_k$ the middle eigenvector of a Rademacher random matrix on the left. Empirical covariance functions $K(s,t)$ at times $s=\frac{1}{4},\frac{1}{2},$ and $\frac{3}{4}$ for $(\theta,\sigma)=(2,1)$ in the center and $(\theta,\sigma)=(20,1)$ on the right.}\label{fig:OU}
\end{figure}

\paragraph{General covariance kernel.} For general covariance kernel, we can use the construction of the proof of Proposition \ref{prop:main2} using the Karhunen--Lo\`eve decomposition of the process. We illustrate this construction using fractional Brownian motion. We note that the usual fractional Brownian motion with Hurst parameter $H$ defined as the centered Gaussian process with covariance 
\[
K_{\mathrm{fBM}}(s,t)=\frac{1}{2}\left(
    t^{2H}+s^{2H}-\vert t-s\vert^{2H}
\right)
\]
does not have, to our knowledge, an explicit Karhunen--Lo\`eve decomposition besides the Brownian motion at $H=\frac{1}{2}.$ In Figure \ref{fig:fBM}, we numerically compute the eigenvalues and eigenfunctions of the process. 
\begin{figure}[!ht]
        \centering
        \includegraphics[height=.19\linewidth]{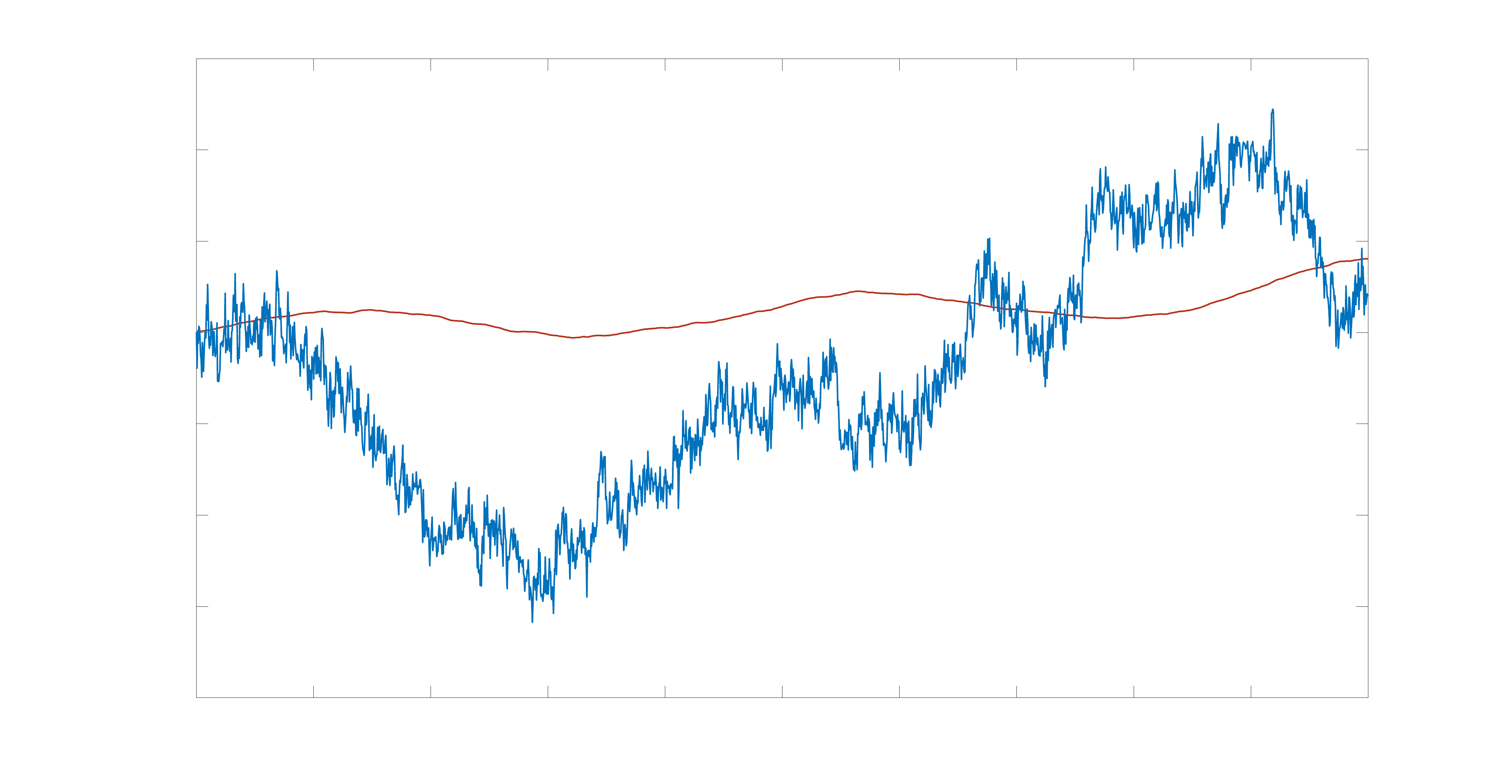}
        \includegraphics[width=.25\linewidth]{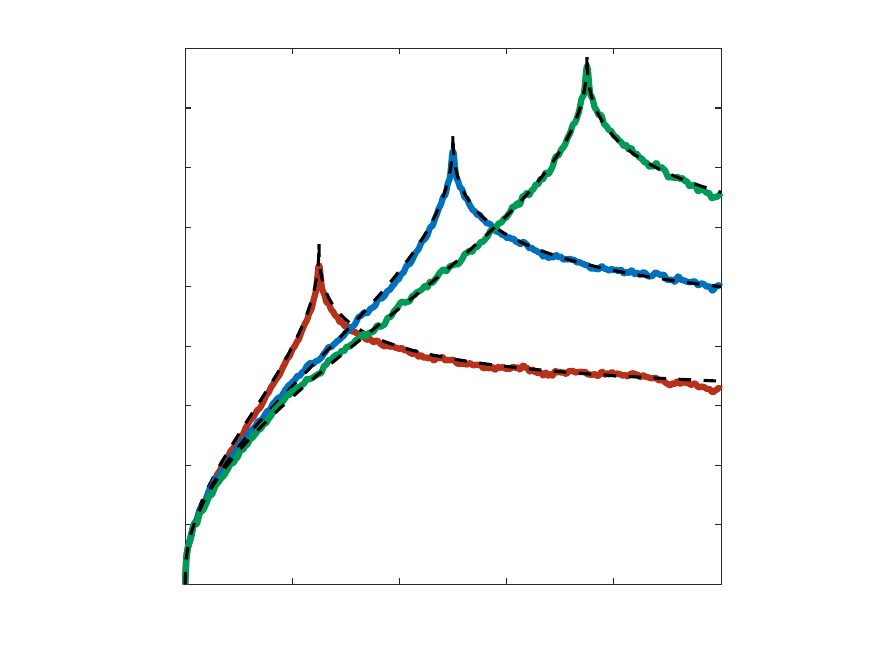}
        \includegraphics[width=.25\linewidth]{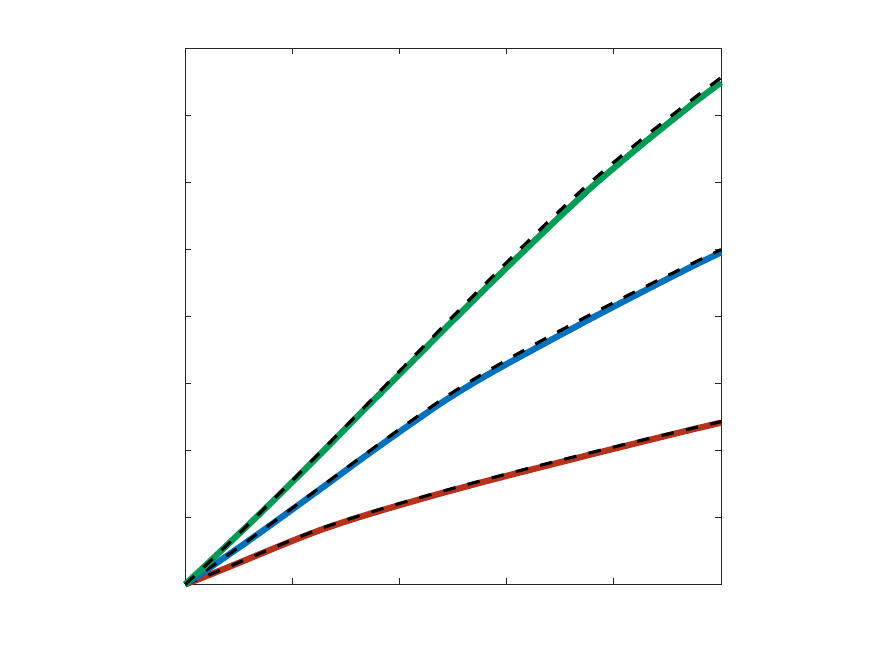}
    \caption{Sample of $\X_{k,k}^\A$ for $\A$ defined using a numerical approximation of the Karhunen--Lo\`eve decomposition of fractional Brownian motion for $H=\frac{1}{5}$ in blue and $H = \frac{9}{10}$ in red for $\u_k$ the middle eigenvector of a Rademacher random matrix on the left. Empirical covariance functions $K(s,t)$ at times $s=\frac{1}{4},\frac{1}{2},$ and $\frac{3}{4}$ for $H = \frac{1}{5}$ in the center and $H = \frac{9}{10}$ on the right.}\label{fig:fBM}
\end{figure}

The first version of fractional Brownian motion, introduced by L\'evy as the Riemann-Liouville integral of a standard Brownian motion $\mathbf{B}$, is given by
\[
B^{H}_t = \frac{1}{\Gamma\left(H+\frac{1}{2}\right)}\int_0^t \left(t-s\right)^{H-\frac{1}{2}}\d B_s
\quad \text{with}\quad
\E\left[
    B^H_tB^H_s
\right]
=
\frac{(s\wedge t)^{2H}}{2H\Gamma\left(H+\frac{1}{2}\right)^2}.
\]
This process has an explicit Karhunen--Lo\`eve decomposition \cite{maccone}. In Figure \ref{fig:RLfBM}, we leverage this explicit decomposition to construct our process as in the proof of Proposition \ref{prop:main2}. 
\begin{figure}[!ht]
        \centering
        \includegraphics[height=.19\linewidth]{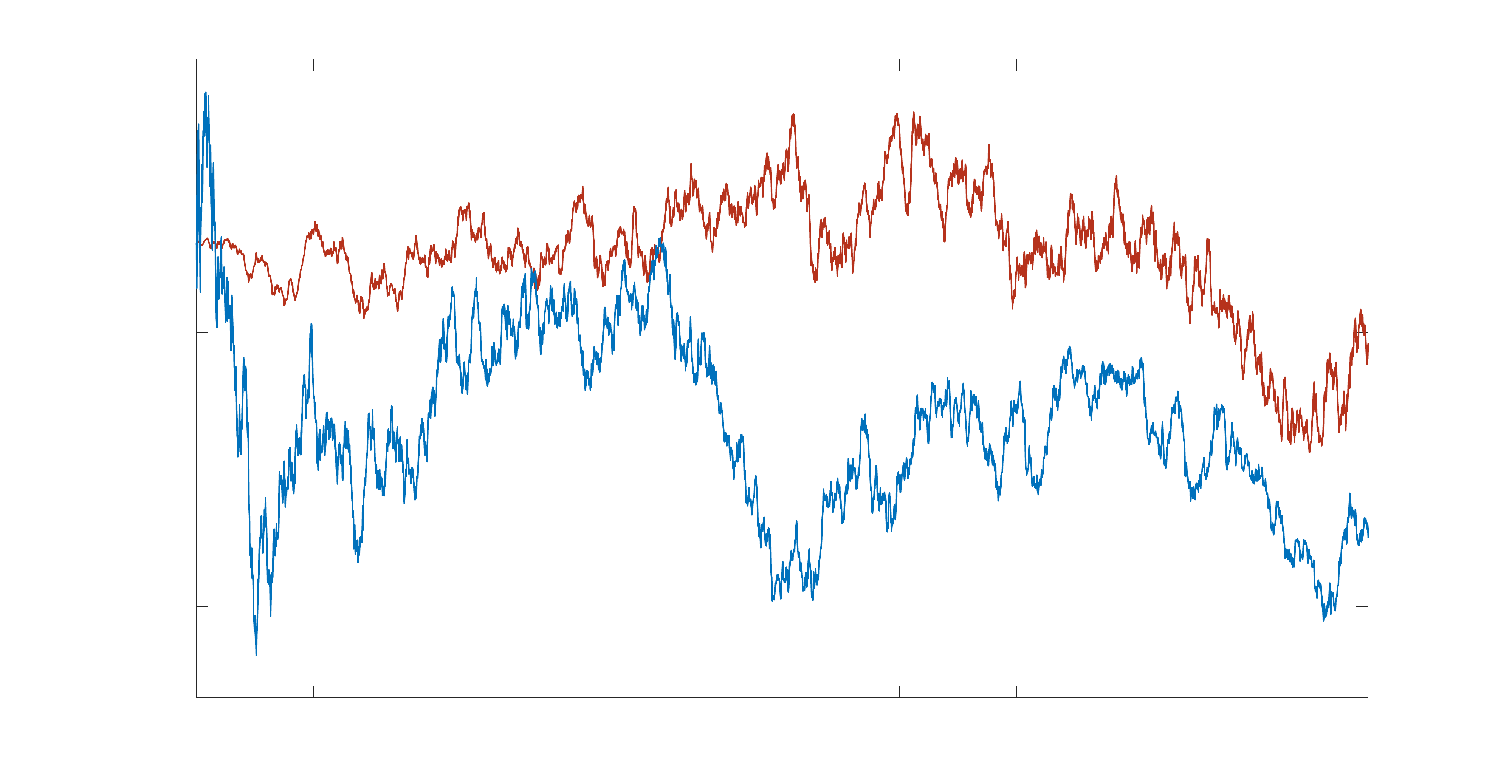}
        \includegraphics[width=.25\linewidth]{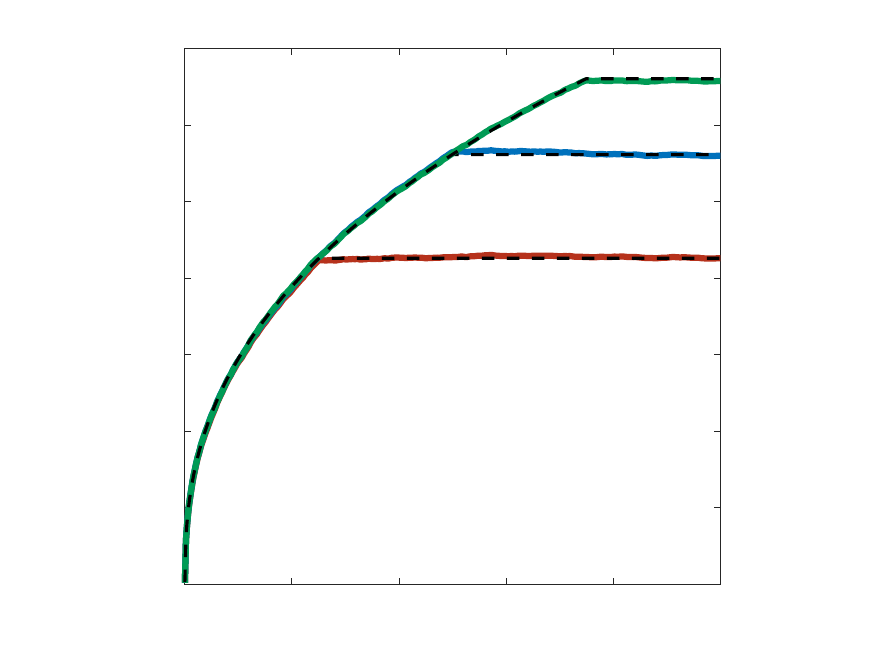}
        \includegraphics[width=.25\linewidth]{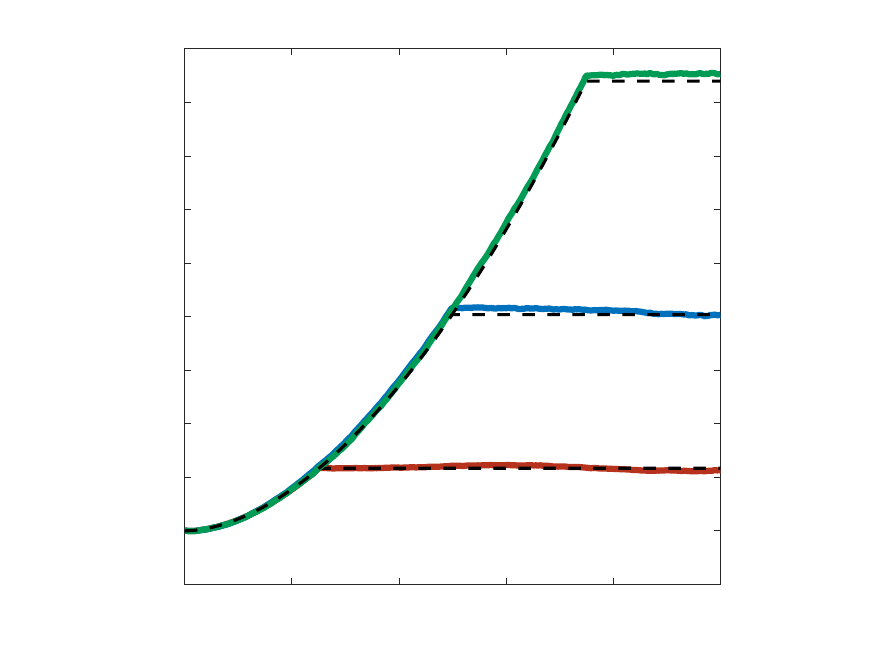}
    \caption{Sample of $\X_{k,k}^\A$ for $\A$ defined using the Karhunen--Lo\`eve decomposition \eqref{eq:KLfBM} for $H=\frac{1}{5}$ in blue and $H = \frac{9}{10}$ in red for $\u_k$ the middle eigenvector of a Rademacher random matrix on the left. Empirical covariance functions $K(s,t)$ at times $s=\frac{1}{4},\frac{1}{2},$ and $\frac{3}{4}$ for $H = \frac{1}{5}$ in the center and $H = \frac{9}{10}$ on the right.}
    \label{fig:RLfBM}
\end{figure}

\noindent We now describe the decomposition of the process. First denote $J_\nu$ the Bessel functions of the first kind with 
\[
\nu = \frac{2H}{2H+1},\quad 
J_\nu(x) = \sum_{n=0}^{+\infty} \frac{(-1)^n}{n!\Gamma(n+\nu+1)}\left(\frac{x}{2}\right)^{2n+\nu}
\] 
and $\gamma_1< \gamma_2<\dots$ defined as $J_{\nu-1}(\gamma_n)=0$, the zeros of Bessel functions. Then the Karhunen--Lo\`eve decomposition of $\mathbf{B}^H$ is given as 
\begin{equation}\label{eq:KLfBM}
\lambda_k^{\mathrm{fBM}} = \frac{1}{\gamma_k^2\Gamma\left(H+\frac{3}{2}\right)^2},\quad \psi^{\mathrm{fBM}}_k(t) = 
\frac{\sqrt{2H+1}\gamma_k}{\sqrt{\gamma_k^2J_\nu'(\gamma_k)^2+(\gamma_k^2-\nu^2)J_\nu(\gamma_k)^2}}
t^{H}J_\nu\left(\gamma_nt^{H+\frac{1}{2}}\right).
\end{equation}

\bibliographystyle{abbrv}
\bibliography{biblio}

\begin{bibdiv}
\begin{biblist}

\bib{adhikari2024eigenstate}{article}{
      author={Adhikari, A.},
      author={Dubova, S.},
      author={Xu, C.},
      author={Yin, J.},
       title={Eigenstate thermalization hypothesis for generalized {W}igner
  matrices},
        date={2024},
        ISSN={1083-6489},
     journal={Electron. J. Probab.},
      volume={29},
       pages={Paper No. 141, 33},
         url={https://doi.org/10.1214/24-ejp1186},
}

\bib{benigni2022fluctuations}{article}{
      author={Benigni, L.},
      author={Cipolloni, G.},
       title={Fluctuations of eigenvector overlaps and the {B}erry conjecture
  for {W}igner matrices},
        date={2024},
     journal={Electron. J. Probab.},
      volume={29},
       pages={1\ndash 19},
}

\bib{benigni2023fluctuations}{article}{
      author={Benigni, L.},
      author={Chen, N.},
      author={Lopatto, P.},
      author={Xie, X.},
       title={Fluctuations in quantum unique ergodicity at the spectral edge},
        date={2023},
     journal={arXiv preprint arXiv:2303.11142},
}

\bib{isotropic2}{article}{
      author={Bloemendal, A.},
      author={Erd\H{o}s, L.},
      author={Knowles, A.},
      author={Yau, H.-T.},
      author={Yin, J.},
       title={Isotropic local laws for sample covariance and generalized
  {W}igner matrices},
        date={2014},
        ISSN={1083-6489},
     journal={Electron. J. Probab.},
      volume={19},
       pages={no. 33, 53},
         url={https://doi.org/10.1214/ejp.v19-3054},
}

\bib{benigni2021fermionic}{article}{
      author={Benigni, L.},
       title={Fermionic eigenvector moment flow},
        date={2021},
     journal={Probab. Theory Related Fields},
      volume={179},
      number={3},
       pages={733\ndash 775},
}

\bib{bourgade2016fixed}{article}{
      author={Bourgade, P.},
      author={Erd{\H{o}}s, L.},
      author={Yau, H.-T.},
      author={Yin, J.},
       title={Fixed energy universality for generalized {W}igner matrices},
        date={2016},
     journal={Comm. Pure Appl. Math.},
      volume={69},
      number={10},
       pages={1815\ndash 1881},
}

\bib{benaych2012universality}{article}{
      author={Benaych-Georges, F.},
       title={A universality result for the global fluctuations of the
  eigenvectors of {W}igner matrices},
        date={2012},
     journal={Random Matrices Theory Appl.},
      volume={1},
      number={04},
       pages={1250011},
}

\bib{billingsley}{book}{
      author={Billingsley, P.},
       title={Convergence of probability measures},
   publisher={John Wiley \& Sons, Inc., New York-London-Sydney},
        date={1968},
}

\bib{benignilopatto}{article}{
      author={Benigni, L.},
      author={Lopatto, P.},
       title={Fluctuations in local quantum unique ergodicity for generalized
  {W}igner matrices},
        date={2022},
        ISSN={0010-3616,1432-0916},
     journal={Comm. Math. Phys.},
      volume={391},
      number={2},
       pages={401\ndash 454},
         url={https://doi.org/10.1007/s00220-022-04314-z},
      review={\MR{4397177}},
}

\bib{benigni2022optimal}{article}{
      author={Benigni, L.},
      author={Lopatto, P.},
       title={Optimal delocalization for generalized {W}igner matrices},
        date={2022},
     journal={Adv. Math.},
      volume={396},
       pages={108109},
}

\bib{bao2014universality}{article}{
      author={Bao, Z.},
      author={Pan, G.},
      author={Zhou, W.},
       title={Universality for a global property of the eigenvectors of
  {W}igner matrices},
        date={2014},
     journal={J. Math. Phys.},
      volume={55},
      number={2},
}

\bib{bourgade2017eigenvector}{article}{
      author={Bourgade, P.},
      author={Yau, H.-T.},
       title={The eigenvector moment flow and local quantum unique ergodicity},
        date={2017},
     journal={Comm. Math. Phys.},
      volume={350},
      number={1},
       pages={231\ndash 278},
}

\bib{bourgade1807random}{article}{
      author={Bourgade, P.},
      author={Yau, H.-T.},
      author={Yin, J.},
       title={Random band matrices in the delocalized phase {I}: {Q}uantum
  unique ergodicity and universality},
        date={2020},
        ISSN={0010-3640,1097-0312},
     journal={Comm. Pure Appl. Math.},
      volume={73},
      number={7},
       pages={1526\ndash 1596},
         url={https://doi.org/10.1002/cpa.21895},
}

\bib{cipolloni2023eigenstate}{article}{
      author={Cipolloni, G.},
      author={Erd{\H{o}}s, L.},
      author={Henheik, J.},
       title={Eigenstate thermalisation at the edge for wigner matrices},
        date={2023},
     journal={arXiv preprint arXiv:2309.05488},
}

\bib{cipolloni2021eigenstate}{article}{
      author={Cipolloni, G.},
      author={Erd{\H{o}}s, L.},
      author={Schr{\"o}der, D.},
       title={Eigenstate thermalization hypothesis for {W}igner matrices},
        date={2021},
     journal={Comm. Math. Phys.},
      volume={388},
      number={2},
       pages={1005\ndash 1048},
}

\bib{cipolloni2022normal}{article}{
      author={Cipolloni, G.},
      author={Erd{\H{o}}s, L.},
      author={Schr{\"o}der, D.},
       title={Normal fluctuation in quantum ergodicity for {W}igner matrices},
        date={2022},
     journal={Ann. Probab.},
      volume={50},
      number={3},
       pages={984\ndash 1012},
}

\bib{cipolloni2022rank}{article}{
      author={Cipolloni, G.},
      author={Erd{\H{o}}s, L.},
      author={Schr{\"o}der, D.},
       title={Rank-uniform local law for {W}igner matrices},
organization={Cambridge University Press},
        date={2022},
     journal={Forum Math. {S}igma},
      volume={10},
       pages={e96},
}

\bib{donati2012truncations}{article}{
      author={Donati-Martin, C.},
      author={Rouault, A.},
       title={Truncations of {H}aar unitary matrices, traces and bivariate
  {B}rownian bridge},
        date={2012},
     journal={Random Matrices Theory Appl.},
      volume={1},
      number={1},
       pages={115007},
}

\bib{erdos2024eigenstate}{article}{
      author={Erd\H{o}s, L.},
      author={Riabov, V.},
       title={Eigenstate thermalization hypothesis for {W}igner-type matrices},
        date={2024},
        ISSN={0010-3616,1432-0916},
     journal={Comm. Math. Phys.},
      volume={405},
      number={12},
       pages={Paper No. 282, 70},
         url={https://doi.org/10.1007/s00220-024-05143-y},
}

\bib{erdHos2009semicircle}{article}{
      author={Erd\H{o}s, L.},
      author={Schlein, B.},
      author={Yau, H.-T.},
       title={Semicircle law on short scales and delocalization of eigenvectors
  for {W}igner random matrices},
        date={2009},
        ISSN={0091-1798,2168-894X},
     journal={Ann. Probab.},
      volume={37},
      number={3},
       pages={815\ndash 852},
         url={https://doi.org/10.1214/08-AOP421},
}

\bib{erdHos2009local}{article}{
      author={Erd{\H{o}}s, L.},
      author={Schlein, B.},
      author={Yau, H.-T.},
       title={Local semicircle law and complete delocalization for {W}igner
  random matrices},
        date={2009},
     journal={Comm. Math. Phys.},
      volume={287},
      number={2},
       pages={641\ndash 655},
}

\bib{erdHos2011universality}{article}{
      author={Erd{\H{o}}s, L.},
      author={Schlein, B.},
      author={Yau, H.-T.},
       title={Universality of random matrices and local relaxation flow},
        date={2011},
     journal={Invent. Math.},
      volume={185},
      number={1},
       pages={75\ndash 119},
}

\bib{erdHos2015gap}{article}{
      author={Erd{\H{o}}s, L.},
      author={Yau, H.-T.},
       title={Gap universality of generalized {W}igner and $\beta$-ensembles},
        date={2015},
     journal={J. Eur. Math. Soc. (JEMS)},
      volume={17},
      number={8},
       pages={1927\ndash 2036},
}

\bib{erdHos2012rigidity}{article}{
      author={Erd{\H{o}}s, L.},
      author={Yau, H.-T.},
      author={Yin, J.},
       title={Rigidity of eigenvalues of generalized {W}igner matrices},
        date={2012},
     journal={Adv. Math.},
      volume={229},
      number={3},
       pages={1435\ndash 1515},
}

\bib{he2025extremal}{article}{
      author={He, Y.},
      author={Huang, J.},
      author={Wang, C.},
       title={Extremal eigenvectors of sparse random matrices},
        date={2025},
     journal={arXiv preprint arXiv:2501.16444},
}

\bib{knowles2013eigenvector}{article}{
      author={Knowles, A.},
      author={Yin, J.},
       title={Eigenvector distribution of {W}igner matrices},
        date={2013},
     journal={Probab. Theory Related Fields},
      volume={155},
      number={3},
       pages={543\ndash 582},
}

\bib{isotropic1}{article}{
      author={Knowles, A.},
      author={Yin, J.},
       title={The isotropic semicircle law and deformation of {W}igner
  matrices},
        date={2013},
        ISSN={0010-3640,1097-0312},
     journal={Comm. Pure Appl. Math.},
      volume={66},
      number={11},
       pages={1663\ndash 1750},
         url={https://doi.org/10.1002/cpa.21450},
}

\bib{maccone}{article}{
      author={Maccone, C.},
       title={Eigenfunction expansion for fractional {B}rownian motions},
        date={1981},
        ISSN={0369-3554,1826-9877},
     journal={Nuovo Cimento B (11)},
      volume={61},
      number={2},
       pages={229\ndash 248},
         url={https://doi.org/10.1007/BF02721326},
}

\bib{marcinek2022high}{article}{
      author={Marcinek, J.},
      author={Yau, H.-T.},
       title={High dimensional normality of noisy eigenvectors},
        date={2022},
     journal={Comm. Math. Phys.},
      volume={395},
      number={3},
       pages={1007\ndash 1096},
}

\bib{silverstein1981describing}{article}{
      author={Silverstein, J.~W.},
       title={Describing the behavior of eigenvectors of random matrices using
  sequences of measures on orthogonal groups},
        date={1981},
     journal={SIAM J. Math. Anal.},
      volume={12},
      number={2},
       pages={274\ndash 281},
}

\bib{tao2011random}{article}{
      author={Tao, T.},
      author={Vu, V.},
       title={Random matrices: universality of local eigenvalue statistics},
        date={2011},
        ISSN={0001-5962,1871-2509},
     journal={Acta Math.},
      volume={206},
      number={1},
       pages={127\ndash 204},
         url={https://doi.org/10.1007/s11511-011-0061-3},
}

\bib{tao2012random}{article}{
      author={Tao, T.},
      author={Vu, V.},
       title={Random matrices: universal properties of eigenvectors},
        date={2012},
     journal={Random Matrices Theory Appl.},
      volume={1},
      number={01},
       pages={1150001},
}

\bib{vu2015random}{article}{
      author={Vu, V.},
      author={Wang, K.},
       title={Random weighted projections, random quadratic forms and random
  eigenvectors},
        date={2015},
     journal={Random Structures Algorithms},
      volume={47},
      number={4},
       pages={792\ndash 821},
}

\end{biblist}
\end{bibdiv}
\end{document}